\documentclass[11pt,reqno]{amsart}

\usepackage{amssymb,amsmath,amsthm}
\usepackage{enumerate,enumitem}
\usepackage[a4paper,hmargin=1.0in,vmargin=1.0in]{geometry}
\usepackage{color}

\newtheorem{theorem}{Theorem}[section]
\newtheorem{lemma}[theorem]{Lemma}
\newtheorem{proposition}[theorem]{Proposition}
\newtheorem{corollary}[theorem]{Corollary}

\newtheorem{claim}[theorem]{Claim}
\newtheorem{question}[theorem]{Question}

\theoremstyle{definition}
\newtheorem{definition}[theorem]{Definition}
\newtheorem{remark}[theorem]{Remark}

\newcommand{\Aut}{\mathrm{Aut}}

\newcommand{\eps}{\varepsilon}

\newcommand{\E}{{\mathbb{E}}}
\newcommand{\Var}{{\bf Var}}

\newcommand{\N}{\mathcal N}
\newcommand{\C}{\mathcal{C}}
\newcommand{\F}{\mathcal F}

\newcommand{\cA}{\mathcal{A}}
\newcommand{\ex}{\mathrm{ex}}
\newcommand{\GG}{\mathcal{G}}

\newcommand{\Free}{\mathrm{Free}}

\newcommand{\HH}{\mathcal{H}}

\newcommand{\cI}{\mathcal{I}}
\newcommand{\cP}{\mathcal{P}}
\newcommand{\cS}{\mathcal{S}}
\newcommand{\T}{\mathcal{T}}
\newcommand{\exx}{\mathrm{ex}^*}

\renewcommand{\le}{\leqslant}
\renewcommand{\ge}{\geqslant}

\begin{document}

\title{A generalized Tur\'an problem in random graphs}

\author{Wojciech Samotij}
\address{School of Mathematical Sciences, Tel Aviv University, Tel Aviv 6997801, Israel}
\email{samotij@tauex.tau.ac.il}

\author{Clara Shikhelman}
\address{School of Mathematical Sciences, Tel Aviv University, Tel Aviv 6997801, Israel}
\email{clara.shikhelman@gmail.com}

\thanks{Research supported in part by the Israel Science Foundation grant 1147/14 (WS) and grants from the Israel Science Foundation and the German-Israeli Foundation for Scientific Research and Development (CS)}

\maketitle

\begin{abstract}
  We study the following generalization of the Tur\'an problem in sparse random graphs. Given graphs $T$ and $H$, let $\ex\big(G(n,p), T, H\big)$ be the random variable that counts the largest number of copies of $T$ in a subgraph of $G(n,p)$ that does not contain $H$. We study the threshold phenomena arising in the evolution of the typical value of this random variable, for every $H$ and an arbitrary $2$-balanced $T$.

  Our results in the case when $m_2(H) > m_2(T)$ are a natural generalization of the Erd\H{o}s--Stone theorem for $G(n,p)$, which was proved several years ago by Conlon and Gowers and by Schacht; the case $T = K_m$ has been recently resolved by Alon, Kostochka, and Shikhelman. More interestingly, the case when $m_2(H) \le m_2(T)$ exhibits a more complex and subtle behavior. Namely, the location(s) of the (possibly multiple) threshold(s) are determined by densities of various coverings of $H$ with copies of $T$ and the typical value(s) of $\ex\big(G(n,p), T, H\big)$ are given by solutions to deterministic hypergraph Tur\'an-type problems that we are unable to solve in full generality.
\end{abstract}

\section{Introduction}

The well-known Tur\'an function is defined as follows. For a fixed graph $H$ and an integer $n$, we let $\ex(n,H)$ be the maximum number of edges in an $H$-free\footnote{A graph is $H$-free if it does not contain $H$ as a (not necessarily induced) subgraph.} subgraph of $K_n$. This function has been studied extensively and generalizations of it were offered in different settings (see~\cite{Si} for a survey). Erd\H{o}s and Stone~\cite{ESt} determined $\ex(n,H)$ for any nonbipartite graph $H$ up to lower order terms.

\begin{theorem}[\cite{ESt}]
  \label{thm:ErdSto}
  For every fixed nonempty graph $H$,
  \[
    \ex(n,H) = \left( 1-\frac{1}{\chi(H)-1}+o(1) \right)\binom{n}{2}.
  \]
\end{theorem}

\noindent
Note that if $H$ is bipartite, then the theorem only tells us that $\ex(n,H)=o(n^2)$. In fact, the classical result of K\H{o}v\'ari, S\'os, and Tur\'an~\cite{KovSosTur54} implies that in this case $\ex(n,H) = O(n^{2-c})$ for some $c > 0$ that depends only on $H$.

Two natural generalizations of Theorem~\ref{thm:ErdSto} have been considered in the literature. First, instead of maximizing the number of edges in an $H$-free subgraph of the complete graph with $n$ vertices, one can consider only $H$-free subgraphs of some other $n$-vertex graph $G$. One natural choice is to let $G$ be the random graph $G(n,p)$, that is, the random graph on $n$ vertices whose each pair of vertices forms an edge independently with probability $p$. This leads to the study of the random variable $\ex\big(G(n,p),H\big)$, the maximum number of edges in an $H$-free subgraph of $G(n,p)$. Considering the intersection between the largest $H$-free subgraph of $K_n$ and the random graph $G(n,p)$, one can show that if $p \gg \ex(n, H)^{-1}$, then w.h.p.\footnote{We write w.h.p.\ as an abbreviation of with high probability, that is, with probability tending to one as the number of vertices $n$ tends to infinity.}
\begin{equation}
  \label{eq:Erdos-Stone-lower-Gnp}
  \ex\big( G(n,p), H\big) \ge (1+o(1)) \cdot \ex(n,H)p.
\end{equation}

The above bound is not always best-possible. If $p$ decays sufficiently fast so that the expected number of copies of (some subgraph of) $H$ that contain a given edge of $G(n,p)$ is $o(1)$, then~\eqref{eq:Erdos-Stone-lower-Gnp} can be strengthened to $\ex\big(G(n,p), H\big) \ge (1+o(1)) \cdot \binom{n}{2}p$. Indeed, one can remove all copies of $H$ from $G(n,p)$ by arbitrarily removing an edge from each copy of (some subgraph $H'$ of) $H$ and the assumption on $p$ implies that w.h.p.\ only a tiny proportion of the edges will be removed this way. Such considerations naturally lead to the notion of \emph{2-density} of $H$, denoted by $m_2(H)$, which is defined by
\[
  m_2(H) = \max\left\{\frac{e_H-1}{v_H-2} : H '\subseteq H, \, e_{H'} \ge 2\right\}.
\]
Moreover, we say that $H$ is \emph{$2$-balanced} if $H$ itself is one of the graphs achieving the maximum above, that is, if $m_2(H) = (e_H-1)/(v_H-2)$. It is straightforward to verify that the expected number of copies of (some subgraph $H'$ of) $H$ that contain a given edge of $G(n,p)$ tends to zero precisely when $p \ll n^{-1/m_2(H)}$. 

Haxell, Kohayakawa, \L uczak, and R\"odl~\cite{HKL95, KLR} conjectured that if the opposite inequality $p \gg n^{-1/m_2(H)}$ holds, then the converse of~\eqref{eq:Erdos-Stone-lower-Gnp} must (essentially) be true. (The case when $H$ is bipartite is much more subtle; see, e.g., \cite{KohKreSte98, MorSax16}.) This conjecture was proved by Conlon and Gowers~\cite{CG}, under the additional assumption that $H$ is $2$-balanced, and, independently, by Schacht~\cite{Sc}; see also \cite{BMS, ConGowSamSch14, FriRodSch10, Sam14, ST}.

\begin{theorem}[\cite{CG,Sc}]
  \label{thm:RegRandTur}
  For any fixed graph $H$ with at least two edges, the following holds w.h.p.\
  \[
    \ex\big(G(n,p), H\big) =
    \begin{cases}
      \left(1-\frac{1}{\chi(H)-1}+o(1)\right) \binom{n}{2} p & \text{if } p \gg n^{-1/m_2(H)}, \\
      (1+o(1)) \cdot \binom{n}{2}p & \text{if } n^{-2} \ll p \ll n^{-1/m_2(H)}.
    \end{cases}
  \]
\end{theorem}

The second generalization of the Tur\'an problem is to fix two graphs $T$ and $H$ and ask to determine the maximum number of copies of $T$ in an $H$-free subgraph of $K_n$. Denote this function by $\ex(n,T,H)$ and note that $\ex(n, H) = \ex(n, K_2, H)$, so this is indeed a generalization. Erd\H{o}s~\cite{E62} resolved this question in the case when both $T$ and $H$ are complete graphs, proving that the balanced complete $(\chi(H)-1)$-partite graph has the most copies of $T$. Another notable result was recently obtained by Hatami, Hladk\'y, Kr\'a\v l, Norine, and Razborov~\cite{HHKN} and, independently, by Grzesik~\cite{G}, who determined $\ex(n, C_5, K_3)$, resolving an old conjecture of Erd\H{o}s. The systematic study of the function $\ex(n, T, H)$ for general $T$ and $H$, however, was initiated only recently by Alon and Shikhelman~\cite{ASh}.

Determining the function $\ex(n,T,H)$ asymptotically for arbitrary $T$ and $H$ seems to be a very difficult task and a generalization of Theorem~\ref{thm:ErdSto} to this broader context has yet to be discovered. On the positive side, a nowadays standard argument can be used to derive the following generalization of the Erd\H{o}s--Stone theorem to the case when $T$ is a complete graph from the aforementioned result of Erd\H{o}s.

\begin{theorem}
  \label{thm:Erdos-Stone-T}
  For any fixed nonempty graph $H$ and any integer $m \ge 2$,
  \[
    \ex(n, K_m, H) = \binom{\chi(H)-1}{m}\left(\frac{n}{\chi(H)-1}\right)^m + o(n^m).
  \]
\end{theorem}

Analogously to Theorem~\ref{thm:ErdSto}, in the case $\chi(H) \le m$, the above theorem only tells us that $\ex(n, K_m, H) = o(n^m)$. The following simple proposition generalizes this fact. A \emph{blow-up} of a graph $T$ is any graph obtained from $T$ by replacing each of its vertices with an independent set and each of its edges with a complete bipartite graph between the respective independent sets.

\begin{proposition}[\cite{ASh}]
  \label{prop:PropositionNotBU}
  Let $T$ be a fixed graph with $t$ vertices. Then $\ex(n,T,H) =\Omega(n^t)$ if and only if $H$ is not a subgraph of  a blow-up of $T$. Otherwise, $\ex(n, T, H) \le n^{t-c}$ for some $c >0$ that depends only on $T$ and $H$.
\end{proposition}

We remark that both the problems of (i)~determining the limit of $\ex(n, T, H) \cdot n^{-t}$ for general $T$ and $H$ such that $H$ is not contained in a blow-up of $T$ and (ii)~computing $\ex(n,T,H)$ up to a constant factor for arbitrary $T$ and $H$ such that $H$ is contained in a blow-up of $T$ seem extremely difficult. Even the case $T = K_2$ of (ii) alone, that is, determining the order of magnitude of the Tur\'an function $\ex(n, H)$ for an arbitrary bipartite graph $H$ is a notorious open problem, see~\cite{FurSim13}.

The common generalization of Theorems~\ref{thm:RegRandTur} and~\ref{thm:Erdos-Stone-T} was considered in~\cite{ASK}. Let $\ex\big(G(n,p),T,H\big)$ be the random variable that counts the maximum number of copies of $T$ in an $H$-free subgraph of $G(n,p)$. Generalizing the easy argument that yields~\eqref{eq:Erdos-Stone-lower-Gnp}, one can show that the inequality
\[
  \ex\big( G(n,p), T, H \big) \ge \left(\ex(n,T,H) + o\big(n^{v_T}\big)\right) p^{e_T}
\] 
holds (w.h.p.) whenever $p \gg n^{-{v_{T'}}/{e_{T'}}}$ for every nonempty $T' \subseteq T$; it is well-known that if $p = O(n^{-{v_{T'}}/{e_{T'}}})$ for some $T' \subseteq T$, then $G(n,p)$ contains no copies of $T$ with probability $\Omega(1)$. It seems natural to guess that the opposite inequality holds as soon as $p \gg n^{-1/m_2(H)}$. The case $T=K_m$ was studied in~\cite{ASK}, where the following generalization of Theorem~\ref{thm:RegRandTur} was proved.

\begin{theorem}[\cite{ASK}]
  \label{thm:m_2InEasyOrder}
  Let $m \ge 2$ be an integer and let $H$ be a fixed graph with $m_2(H) > m_2(K_m)$ and $\chi(H) > m$. If $p$ is such that $\binom{n}{m}p^{\binom{m}{2}}$ tends to infinity with $n$, then w.h.p.\
  \[
    \ex\big(G(n,p), K_m, H\big)=
    \begin{cases}
      (1+o(1)) \cdot \binom{\chi(H)-1}{m} \left(\frac{n}{\chi(H)-1}\right)^m p^{\binom{m}{2}}& \text{if } p\gg n^{ -1/m_2(H)}, \\
      (1+o(1)) \cdot \binom{n}{m}p^{\binom{m}{2}} & \text{if }p\ll n^{ -1/m_2(H)}.
    \end{cases}
  \]
\end{theorem}

Let us draw the reader's attention to the assumption that $m_2(H) > m_2(K_m)$ in the statement of the theorem. No such assumption was (explicitly) present in the statement of Theorem~\ref{thm:RegRandTur} and it is natural to wonder whether it is really necessary. Since we assume that $H$ is not $m$-colorable, then it must contain a subgraph whose average degree is at least $m$, larger than the average degree of $K_m$. In particular, it is natural to guess that this implies that the $2$-density of $H$ is larger than the $2$-density of $K_m$. Perhaps surprisingly, this is not true and only the weaker inequality $m_2(H) > m_2(K_{m-1})$ does hold for every non-$m$-colorable graph $H$. A construction of a graph $H$ such that $\chi(H)=4$ and $m_2(H)<m_2(K_3)$ was given in~\cite{ABGKR}. Subsequently, constructions of graphs $H$ such that $\chi(H)=m+1$ but $m_2(H) < m_2(K_m)$ were given for all $m$ in \cite{ASK}. It was also shown there that for such graphs $H$, the typical value of $\ex\big(G(n,p), K_m, H\big)$ does not change at $p=n^{-1/m_2(H)}$, as in Theorem~\ref{thm:m_2InEasyOrder}. More precisely, if $p=n^{-1/m_2(H)+\delta}$ for some small but fixed $\delta = \delta(H) > 0$, then still $\ex\big(G(n,p),K_m,H\big)=(1+o(1)) \cdot \binom{n}{m}p^{\binom{m}{2}}$. This led to the following open questions:
\begin{enumerate}[label=(\roman*)]
\item
  Where does the `phase transition' of $\ex\big(G(n,p), K_m, H\big)$ take place if $m_2(H) \le m_2(K_m)$?
\item
  \label{item:exGnpTH-growth}
  How does the function $p \mapsto \ex\big(G(n,p), T, H\big)$ grow for general $T$ and $H$?
\end{enumerate}
In this paper we answer both of these questions under the assumptions that $T$ is $2$-balanced and $H$ is not contained in a blow-up of $T$. Answering question~\ref{item:exGnpTH-growth} in the case when $H$ is contained in a blow-up of $T$ seems extremely challenging, as even the order of magnitude of $\ex(n, T, H)$, which corresponds to setting $p=1$ above, is not known for general graphs $T$ and $H$, see the comment below Proposition~\ref{prop:PropositionNotBU}. 

The case when $m_2(H) > m_2(T)$ holds no surprises, as the following extension of Theorem~\ref{thm:m_2InEasyOrder} is valid. We denote by $\N_T(K_n)$ the number of copies of a graph $T$ in the complete graph $K_n$.

\begin{theorem}
  \label{thm:main-easy}
  If $H$ and $T$ are fixed graphs such that $T$ is $2$-balanced and that $m_2(H) > m_2(T)$, then w.h.p.\
  \[
    \ex\big( G(n,p),T,H \big) =
    \begin{cases}
      \left( \N_T(K_n) + o\big(n^{v_T}\big)\right) p^{e_T} & \text{if } n^{-v_T/e_T} \ll p\ll n^{ -1/m_2(H)},\\
      \left( \ex\big(n,T,H\big) + o\big(n^{v_T}\big)\right) p^{e_T} & \text{if } p\gg n^{ -1/m_2(H)}.
    \end{cases}
  \]
\end{theorem}

As already hinted at by~\cite{ASK}, the case when $m_2(H) \le m_2(T)$ exhibits a more complex behavior. We find that there are several potential `phase transitions' and we relate their locations to a measure of density of various coverings of $H$ with copies of $T$ that generalizes the notion of the $2$-density of $H$. Moreover, we show that the (typical) asymptotic value of $\ex\big(G(n,p), T, H\big)$ is determined, for every $p$ that does not belong to any of the constantly many `phase transition windows', by a solution of a deterministic hypergraph Tur\'an-type problem. Unfortunately, we were unable to solve this Tur\'an-type problem in full generality. Worse still, we do not understand it sufficiently well to either show that for some pairs of $T$ and $H$, the function $p \mapsto \ex\big(G(n,p), T, H\big)$ undergoes more than one `phase transition' or to rule out the existence of such pairs. We leave these questions as a challenge for future work.

In order to make the above discussion formal and state the main theorem, we will require several definitions.

\subsection{Notations and definitions}
\label{sec:notat-defin}

A \emph{$T$-covering} of $H$ is a minimal collection $F = \{T_1, \dotsc, T_k\}$ of pairwise edge-disjoint copies of $T$ (in a large complete graph) whose union contains a copy of~$H$.\footnote{The collection $F = \{T_1, \dotsc, T_k\}$ is minimal in the sense that for every $i \in [k]$, the union of all graphs in $F \setminus \{T_i\}$ no longer contains a copy of $H$.} Given two $T$-coverings $F = \{T_1, \dotsc, T_k\}$ and $F' = \{T_1', \dotsc, T_k'\}$, a map $f$ from the union of the vertex sets of the $T_i$s to the union of the vertex sets of the $T_i'$s is an \emph{isomorphism} if it is a bijection and for every $T_i \in F$, the graph $f(T_i)$ belongs to $F'$. We can then say that the \emph{type} of a $T$-covering of $H$ is just the isomorphism class of this covering. Observe that there are only finitely many types of $T$-coverings of $H$. One special type of a $T$-covering of $H$ that will be important in our considerations is the covering of $H$ with $e_H$ copies of $T$ such that each copy of $T$ intersects $H$ in a single edge and is otherwise completely (vertex) disjoint from the remaining $e_H-1$ copies of $T$ that constitute this covering. We denote this covering by $F_{T,H}^e$ and note that the union of all members of $F_{T,H}^e$ is a graph with $v_H + e_H (v_T-2)$ vertices and $e_H e_T$ edges.

For a collection $F'$ of copies of $T$, denote by $U(F')$ the \emph{underlying graph} of $F'$, that is, the union of all members of $F'$. We define the \emph{$T$-density} of a $T$-covering $F$, which we shall denote by $m_T(F)$, as follows:
\[
  m_T(F) = \max\left\{\frac{e_{U(F')}-e_T}{v_{U(F')}-v_T}: F' \subseteq F, |F'| \ge 2\right\}.
\]
Note that this generalizes the notion of $2$-density of a graph. Indeed, the $2$-density of $H$ is the $K_2$-density of (the edge set of) $H$. The notion of $T$-density is motivated by the following observation. For graphs $G$ and $T$, we let $T(G)$ denote the collection of copies of $T$ in $G$ and let $\N_T(G) = |T(G)|$.

\begin{remark}
  \label{remark:m-T-F}
  For every collection $F$ of at least two copies of $T$,
  \[
    \E\left[\N_{U(F')}\big(G(n,p)\big)\right] \ll \E\left[\N_T\big(G(n,p)\big)\right] \text{ for some $F' \subseteq F$} \quad \Longleftrightarrow \quad p \ll n^{-1/m_T(F)}.
  \]  
\end{remark}

Even though we are interested in maximizing $\N_T(G)$ in an $H$-free subgraph $G \subseteq K_n$, we shall be considering (more general) abstract collections of $T$-copies in $K_n$ that do not contain a $T$-covering of $H$ of a certain type (or a set of types). In particular, if $G \subseteq K_n$ is $H$-free, then $T(G)$ is one such collection of $T$-copies, as it does not contain any $T$-covering of $H$ (since the underlying graph of every $T$-covering of $H$ contains $H$ as a subgraph). However, not all the collections we shall consider will be `graphic', that is, of the form $T(G)$ for some graph $G$. 

The aforementioned Tur\'an-type problem for hypergraphs asks to determine the following quantity. For a given family $\F$ of $T$-coverings of $H$, we let $\exx(n,T,\F)$ be the maximum size of a collection of copies of $T$ in $K_n$ that does not contain any member of $\F$. Note that for any collection $\F$ of $T$-coverings of $H$, one has that $\exx(n,T,\F) \ge \ex(n, T, H)$. Indeed, if $G$ is an $H$-free graph with $n$ vertices such that $\ex(n,T,H)=\N_T(G)$, then $T(G)$ is $\F$-free. However, this inequality can be strict, as not every collection of $T$-copies is of the form $T(G)$ for some graph $G$. Having said that, we shall show in Lemma~\ref{lem:Fe-exnTH} that at least $\exx(n,T,F_{T,H}^e) = \ex(n,T,H) + o(n^{v_T})$. We are now equipped to formulate the key definition needed to state our main result.

\begin{definition}
  \label{def:T-resolution}
  Suppose that $T$ and $H$ are fixed graphs and assume that $T$ is $2$-balanced. The \emph{$T$-resolution} of $H$ is the sequence $F_1, \dotsc, F_k$ of all types of $T$-coverings of $H$ whose $T$-density does not exceed $m_T(F_{T,H}^e)$, ordered by their $T$-density (with ties broken arbitrarily). The \emph{associated threshold sequence} is the sequence $p_0, p_1, \dotsc, p_k$, where $p_0 = n^{-v_T / e_T}$ and $p_i = n^{-1/m_T(F_i)}$ for $i \in [k]$.
\end{definition}

\subsection{Statement of the main theorem}

The following theorem is the main result of this paper.

\begin{theorem}
  \label{thm:main}
  Suppose that $H$ and $T$ are fixed graphs and assume that $T$ is $2$-balanced and that $m_2(H) \le m_2(T)$. Let $F_1, \dotsc, F_k$ be the $T$-resolution of $H$ and let $p_0, p_1, \dotsc, p_k$ be the associated threshold sequence. Then the following hold for every $i \in [k]$:
  \begin{enumerate}[label=(\roman*)]
  \item
    \label{item:main-0-statement}
    If $p_0 \ll p \ll p_i$, then w.h.p.\
    \[
      \ex\big(G(n,p),T,H\big) \ge \left( \exx\big(n,T,\{F_1, \dotsc, F_{i-1}\}\big) + o\big(n^{v_T}\big)\right) p^{e_T}.
    \]
  \item
    \label{item:main-1-statement}
    If $ p\gg p_i$, then w.h.p.\
    \[
      \ex\big(G(n,p),T,H\big) \le \left( \exx\big(n,T,\{F_1, \dotsc, F_i\}\big) + o\big(n^{v_T}\big)\right) p^{e_T}.
    \]
  \end{enumerate}
\end{theorem}

Even though the above theorem determines the typical values of $\ex\big(G(n,p), T, H\big)$ for almost all $p$, these values remain somewhat of a mystery as we do not know how to compute $\exx\big(n, T, \{F_1, \dotsc, F_i\}\big)$ in general. One thing that we do know how to prove is that $\exx(n, T, \F) = \ex(n, T, H) + o(n^{v_T})$ for every family $\F$ of $T$-coverings of $G$ that contains the special covering $F_{T,H}^e$, see Lemma~\ref{lem:Fe-exnTH}. Moreover, it is not hard to verify that
\[
  m_T(F_{T,H}^e) = \frac{e_T}{v_T-2+1/m_2(H)},
\]
which, when $T$ is $2$-balanced, is equal to the so-called \emph{asymmetric $2$-density of $T$ and $H$}, a quantity that arises in the study of asymmetric Ramsey properties of $G(n,p)$, see~\cite{GugNenPerSkoSteTho17, KohKre97, KohSchSpo14, MouNenSam18}. Note that if $T$ is $2$-balanced and $m_2(H) < m_2(T)$, then $m_2(H) < m_T(F_{T,H}^e) < m_2(T)$. An `abbreviated' version of Theorem~\ref{thm:main} can be now stated as follows.

\begin{corollary}
  \label{cor:main}
  Suppose that $H$ and $T$ are fixed graphs and assume that $T$ is $2$-balanced and that $m_2(H) \le m_2(T)$. There is an integer $t \ge 1$ and rational numbers $\mu_0 < \dotsc < \mu_t$, where
  \[
    \mu_0 = \frac{e_T}{v_T} \qquad \text{and} \qquad \mu_k \le \frac{e_T}{v_T - 2 + 1/m_2(H)},
  \]
  and real numbers $\pi_0 > \dotsc > \pi_t$, where
  \[
    \pi_0 = \frac{1}{|\Aut(T)|} \qquad \text{and} \qquad \pi_t = \lim_{n \to \infty} \ex(n, T, H) \cdot n^{-v_T},
  \]
  such that w.h.p.
  \[
    \ex\big(G(n,p),T,H\big) =
    \begin{cases}
      (\pi_i + o(1)) n^{v_T}p^{e_T}, & n^{-1/\mu_i} \ll p \ll n^{-1/\mu_{i+1}} \quad \text{for $i \in \{0, \dotsc, t-1\}$}, \\
      (\pi_t + o(1)) n^{v_T}p^{e_T}, & p \gg n^{-1/\mu_t}.
    \end{cases}
  \]
\end{corollary}

A rather disappointing feature of Corollary~\ref{cor:main} (and thus of Theorem~\ref{thm:main}) is that we are unable to determine whether or not there exists a pair of graphs $H$ and $T$ for which the typical value of $\ex\big(G(n,p), T, H\big)$ undergoes more than one `phase transition' (that is, the integer $t$ from the statement of the corollary is strictly greater than one). If one was allowed to replace $H$ with a finite family of forbidden graphs, then one can see an arbitrary (finite) number of `phase transitions' even in the case when $T = K_2$, see~\cite[Theorem~6.4]{MorSamSax}.

Even though we were able to construct pairs of $H$ and $T$ which admit $T$-coverings of $H$ whose $T$-density is strictly smaller than the $T$-density of the special covering of $H$ with $e_H$ copies of $T$, for no such $T$-covering $F$ we were able to show that $\exx(n, T, F) \ge \ex(n, T, H) + \Omega(n^{v_T})$. On the other hand, if one removes the various (important) assumptions on the densities of $H$, $T$, and $F$, then one can find such triples. A simple example is $H = K_7$, $T = K_3$, and $F$ being a decomposition of $K_7$ into edge-disjoint triangles (the Fano plane). Indeed, in this case $\exx(n, K_3, F) \ge (\frac{3}{4}-o(1))\binom{n}{3}$ as witnessed by the family of all triangles in $K_n$ that have at least one vertex in each of the parts of some partition of $V(K_n)$ into two sets of (almost) equal size (the Fano plane is not $2$-colorable). On the other hand, Theorem~\ref{thm:Erdos-Stone-T} implies that $\ex(n, K_3, K_7) \le (\frac{5}{9}+o(1))\binom{n}{3}$. We thus pose the following question.

\begin{question}
  \label{question:main}
  Do there exist pairs of graphs $H$ and $T$ such that $m_2(H) \le m_2(T)$, $T$ is $2$-balanced, and the family $\F$ of all $T$-coverings of $H$ that have the smallest $T$-density (among all $T$-coverings of $H$) satisfies $\exx(n, T, \F) \ge \ex(n, T, H) + \Omega(n^{v_T})$?
\end{question}

Let us point out that answering Question~\ref{question:main} is equivalent to determining whether or not there is a pair of graphs $H$ and $T$, where $T$ is $2$-balanced, for which $\ex\big( G(n,p), T, H \big)$ undergoes multiple `phase transitions' in the sense described above. Indeed, suppose that $H$ and $T$ are fixed graphs and assume that $T$ is $2$-balanced and that $m_2(H) \le m_2(T)$. Let $F_1, \dotsc, F_k$ be the $T$-resolution of $H$ and let $p_0, p_1, \dotsc, p_k$ be the associated threshold sequence. The numbers $\pi_0, \pi_1, \dotsc, \pi_t$ from the statement of Corollary~\ref{cor:main} are precisely all numbers $\pi$ satisfying
\[
  \pi = \lim_{n \to \infty} \exx(n, T, \{F_1, \dotsc, F_i\}) \cdot n^{-v_T}
\]
for some $i \in \{0, \dotsc, k\}$ such that either $p_{i+1} \neq p_i$ or $i = k$. Standard averaging arguments can be used to show that $\exx(n,T, \F) \le \N_T(K_n) - \Omega(n^{v_T})$ for every nonempty family $\F$ of $T$-coverings of $H$ whereas the aforementioned Lemma~\ref{lem:Fe-exnTH} yields
\[
  \ex(n, T, H) \le \exx(n, T, \{F_1, \dotsc, F_k\}) \le \exx(n, T, F_{T,H}^e) \le \ex(n, T, H) + o(n^2).
\]
Thus $t > 1$ is and only if there exists some $i \in \{1, \dotsc, k-1\}$ such that
\[
  p_{i+1} \gg p_i \qquad \text{and} \qquad \exx(n, T, \{F_1, \dotsc, F_i\}) \ge \ex(n, T, H) + \Omega(n^{v_T}).
\]
If the latter condition is satisfied, then it also holds when $i$ is the largest index such that $p_1 = p_i$. But then $\{F_1, \dotsc, F_i\}$ is precisely the family $\F$ defined in Question~\ref{question:main}.

\subsection{Structure of the paper}

The rest of the paper is structured as follows. In Section~\ref{sec:proof-outline}, we give a high level overview of the proofs of Theorems~\ref{thm:main-easy} and~\ref{thm:main}. In Section~\ref{sec:Tools}, we introduce the main tools, the hypergraph container lemma and Harris's and Janson's inequalities, and prove a few useful lemmas and corollaries concerning extremal and random graphs. In Section~\ref{sec:proof-main-thms}, we give the proofs of the main theorems, starting with the simpler Theorem~\ref{thm:main-easy} and then continuing to the more difficult Theorem~\ref{thm:main}.
Finally, in Section~\ref{sec:concl-remarks}, we give concluding remarks and offer open problems.

\section{Proof outline}

\label{sec:proof-outline}

Before diving into the details of the proofs of Theorems~\ref{thm:main-easy} and~\ref{thm:main}, let us briefly go over the main steps we take, highlighting the main ideas.

The proofs of the lower bounds on $\ex\big(G(n,p), T, H\big)$ are rather standard. Suppose that $G \sim G(n,p)$. In the setting of Theorem~\ref{thm:main-easy}, the $H$-free subgraph of $G$ with a large number of copies of $T$ is obtained by arbitrarily removing from $G$ one edge from every copy of (some subgraph of) $H$. In the setting of Theorem~\ref{thm:main}, we remove from $G$ all edges that are either (i)~not contained in a copy of $T$ that belongs to a fixed extremal $\{F_1, \dotsc, F_{i-1}\}$-free collection $\T \subseteq T(K_n)$, or (ii)~contained in a copy of $T$ that constitutes some $T$-covering of $H$ in $T(G) \cap \T$, or (iii)~contained in more than one copy of $T$ in $G$. Note that all copies of $H$ in $G$ are removed this way. Our assumption on $p$ guarantees that in steps (ii) and (iii) above we lose only a negligible proportion of $T(G)$.

The upper bound implicit in the equality $\ex\big(G(n,p), T, H\big) = \left(\N_T(K_n) + o\big(n^{v_T}\big)\right) p^{e_T}$ in Theorem~\ref{thm:main-easy} follows from a standard application of the second moment method; see, e.g., \cite{JLR-book}. The proofs of the remaining upper bounds, in both Theorems~\ref{thm:main-easy} and~\ref{thm:main}, utilize the method of hypergraph containers~\cite{BMS, ST}; see also~\cite{BMS-survey}. Roughly speaking, the hypergraph container theorems state that the family of all independent sets of a uniform hypergraph whose edges are distributed somewhat evenly can be covered by a relatively small family of subsets, called \emph{containers}, each of which is `almost independent' in the sense that it contains only a negligible proportion of the edges of the hypergraph.

In the setting of Theorem~\ref{thm:main-easy}, a standard application of the method yields a collection $\C$ of $\exp\big(O(n^{2-1/m_2(H)} \log n)\big)$ subgraphs of $K_n$ (the containers), each with merely $o(n^{v_H})$ copies of $H$, that cover the family of all $H$-free subgraphs of $K_n$. Suppose that $G \sim G(n,p)$ and let $G_0$ be an $H$-free subgraph of $G$ and note that $G_0$ has to be a subgraph of one of the containers. A standard supersaturation result states that each graph in $\C$ can have at most $\ex(n,T,H) + o(n^{v_H})$ copies of $H$. It follows that for each \emph{fixed} container $C \in \C$, the intersection of $G$ with $C$ can have no more than $\big(\ex(n,T,H) + o(n^{v_T})\big) p^{e_T}$ copies of $T$. At this point, one would normally take the union bound over all containers and conclude that w.h.p.\ the number of copies of $T$ in $G \cap C$ is small simultaneously for all $C \in \C$ and hence also $G_0$ has this property, as $G_0 \subseteq G \cap C$ for some $C \in \C$.

Unfortunately, we cannot afford to take such a union bound as the rate of the upper tail of the number of copies of $T$ in $G(n,p)$ is much too slow to allow this, see~\cite{JOR}. Luckily, the rate of the \emph{lower tail} of the number of copies of $T$ in $G(n,p)$ is sufficiently fast, see Lemma~\ref{lem:NumberOfCopiesOfT}, to allow a union bound over all containers. Therefore, what we do is first prove that w.h.p.\ $\N_T(G) = (1+o(1)) \N_T(K_n) p^{e_T}$ and then show that w.h.p.\ the number of copies of $T$ in $G$ that are not fully contained in $C$ is at least $\big(\N_T(K_n) - \ex(n,T,H) - o(n^{v_T}) \big)p^{e_T}$ simultaneously for all $C \in \C$. This implies that w.h.p.\ $\N_T(G_0) \le \max_{C \in \C} \N_T(G \cap C) \le \big(\ex(n,T,H) + o(n^{v_T})\big)p^{e_T}$.

In the setting of Theorem~\ref{thm:main}, instead of building containers for all possible graphs $G_0$, we build containers for all possible collections $T(G_0)$, exploiting the fact that $T(G_0)$ cannot contain any $T$-covering of $H$, as $G_0$ is $H$-free. More precisely, we work with hypergraphs $\HH_1, \dotsc, \HH_i$, each with vertex set $T(K_n)$, whose edges are copies of the $T$-coverings $F_1, \dotsc, F_i$, respectively. A version of the container theorem presented in Corollary~\ref{cor:Cont} provides us with a small collection $\C$ of subsets of $T(K_n)$ such that (i)~each $\{F_1, \dotsc, F_i\}$-free collection $\T \subseteq T(K_n)$ is contained in some member of $\C$ and (ii)~each $C \in \C$ has only $o(n^{v_{U(F_j)}})$ copies of $F_j$, for each $j \in [i]$, and thus (by a standard averaging argument) it comprises at most $\exx(n, T, \{F_1, \dotsc, F_i\}) + o(n^{v_T})$ copies of $T$. The key parameter $q$ from the statement of Corollary~\ref{cor:Cont}, which determines the size of $\C$, exactly matches our definitions of $m_T(F_1), \dotsc, m_T(F_i)$. Now, since the underlying graph of each $F_j$ contains $H$ as a subgraph and $G_0$ is $H$-free, $T(G_0)$ must be contained in some member of $\C$. The rest of the argument is similar to the proof of Theorem~\ref{thm:main-easy} -- we first bound the upper tail of $\N_T(G)$ and then the lower tail of $|T(G) \setminus C|$ for all $C \in \C$ simultaneously.

\section{Tools and Preliminary Results}
\label{sec:Tools}

\subsection{Hypergraph container lemma}
\label{sec:hypergraph-containers}

The first key ingredient in our proof is the following version of the hypergraph container lemma, proved by Balogh, Morris, and Samotij~\cite{BMS}. An essentially equivalent statement was obtained independently by Saxton and Thomason~\cite{ST}. We first introduce the relevant notions. Suppose that $\HH$ is a $k$-uniform hypergraph. For a set $B \subseteq V(\HH)$, we let $\deg_{\HH}(B) = \big|\{A\in E(\HH) : B \subseteq A\}\big|$ and for each $\ell \in [k]$, we let
\begin{equation*}
  \Delta_\ell(\HH) = \max\big\{\deg_{\HH}(B) : B \subseteq V(\HH), \, |B|=\ell\big\}.
\end{equation*}
We denote by $\cI(\HH)$ the collection of independent sets in $\HH$.

\begin{theorem}[\cite{BMS}]
  \label{thm:Cont}
  For every positive integer $k$ and all positive $K$ and $\eps$, there exists a positive constant $C$ such that the following holds. Let $\HH$ be a $k$-uniform hypergraph and assume that $q \in (0,1)$ satisfies
  \begin{equation}
    \label{eq:Delta-ell-ass}
    \Delta_\ell(\HH) \le K q^{\ell-1}\frac{e(\HH)}{v(\HH)} \qquad \text{for all $\ell \in [k]$}.
  \end{equation}
  There exist a family $\cS \subseteq \binom{V(\HH)}{\le Cqv(\HH)}$ and functions $f \colon \cS \to \cP(V(\HH))$ and $g \colon \cI(\HH) \to \cS$ such that:
  \begin{enumerate}[label=(\roman*),itemsep=3pt]
  \item
    \label{item:signature-container}
    For every $I \in \cI(\HH)$, $g(I) \subseteq I$ and $I \setminus g(I) \subseteq  f(g(I))$.
  \item
    \label{item:container-sparse}
    For every $S\in \mathcal{S}$, $e(\HH[f(S)]) \le \eps e(\HH)$.
  \item
    \label{item:signature-greedy}
    If $g(I)\subseteq I'$ and $g(I')\subseteq I$ for some $I,I'\in \cI(\HH)$, then $g(I)=g(I')$.
  \end{enumerate}
\end{theorem}

Let us make two remarks here. First, condition~\ref{item:container-sparse} in the above statement is equivalent to the condition that the image of the function $f$ from the statement of \cite[Theorem~2.2]{BMS} is $\overline{\F}$, where $\F$ is the (increasing) family of all subsets of $V(\HH)$ that induce more than $\eps e(\HH)$ edges. Second, that the final assertion of the statement of Theorem~\ref{thm:Cont} is not present in the original statement of~\cite[Theorem~2.2]{BMS}. It is, however, proved in the final claim of the proof of~\cite[Theorem~2.2]{BMS}.

Since the hypergraphs we shall be working with in the proof of Theorem~\ref{thm:main} are not necessarily uniform, we shall be actually invoking the following (rather straightforward) corollary of Theorem~\ref{thm:Cont}.

\begin{corollary}
  \label{cor:Cont}
  For all positive integers $k_1, \dotsc, k_i$ and all positive $K$ and $\eps$, there exists a positive constant $C$ such that the following holds. Suppose that $\HH_1, \dotsc, \HH_i$ are hypergraphs with the same vertex set $V$ and that $\HH_j$ is $k_j$-uniform, for each $j \in [i]$. Assume that $q \in (0,1)$ is such that for all $j \in [i]$,
  \begin{equation}
    \label{eq:Delta-ell-ass-cor}
    \Delta_\ell(\HH_j) \le K q^{\ell-1}\frac{e(\HH_j)}{v(\HH_j)} \qquad \text{for all $\ell \in [k_j]$}.
  \end{equation}
  There exist a family $\cS \subseteq \binom{V}{\le Cq|V|}$ and functions $f \colon \cS \to \cP(V)$ and $g \colon \bigcap_{j=1}^i \cI(\HH_j) \to \cS$ such that:
  \begin{enumerate}[label=(\roman*),itemsep=3pt]
  \item
    For every $I \in \bigcap_{j=1}^i \cI(\HH_j)$, $g(I) \subseteq I$ and $I \setminus g(I) \subseteq  f(g(I))$.
  \item
    For every $S\in \mathcal{S}$, $e(\HH_j[f(S)]) \le \eps e(\HH_j)$ for every $j \in [i]$.
  \item
    If $g(I)\subseteq I'$ and $g(I')\subseteq I$ for some $I,I'\in \bigcap_{j=1}^i \cI(\HH_j)$, then $g(I)=g(I')$. 
  \end{enumerate}
\end{corollary}
\begin{proof}
  For each $j \in [i]$, let $C_j$ be the constant given by Theorem~\ref{thm:Cont} with $k \leftarrow k_j$. Assume that $q \in (0,1)$ is such that the hypergraphs $\HH_1, \dotsc, \HH_i$ satisfy~\eqref{eq:Delta-ell-ass-cor}. For each $j \in [i]$, we may apply Theorem~\ref{thm:Cont} to the hypergraph $\HH_j$ to obtain a family $\cS_j \subseteq \binom{V}{\le C_j |V|}$ and functions $f_j \colon \cS_j \to \cP(V)$ and $g_j \colon \cI(\HH_j) \to \cS_j$ as in the assertion of the theorem.

  We now let $C = C_1 + \dotsc + C_i$ and define
  \[
    \cS = \big\{S_1 \cup \dotsc \cup S_i : S_j \in \cS_j \text{ for each } j \in [i]\big\} \subseteq \binom{V}{\le Cq|V|}
  \]
  and, given an $I \in \bigcap_{j=1}^i \cI(\HH_j)$,
  \[
    g(I) = g_1(I) \cup \dotsc \cup g_i(I).
  \]
  Suppose that $g(I) \subseteq I'$ and $g(I') \subseteq I$ for some $I, I' \in \bigcap_{j=1}^i \cI(\HH_j)$. Then also $g_j(I) \subseteq g(I) \subseteq I'$ and, similarly, $g_j(I') \subseteq g(I') \subseteq I$ for each $j \in [i]$. Assertion~\ref{item:signature-greedy} of Theorem~\ref{thm:Cont} implies that that $g_j(I) = g_j(I')$ for each $j$ and thus $g(I) = g(I')$. Since $g(I) \subseteq I$, we may also conclude that if $g(I) = g(I')$, then also $g_j(I) = g_j(I')$ for each $j \in [i]$. In particular, we may define, for each $I \in \bigcap_{j=1}^i \cI(\HH_j)$,
  \[
    f(g(I)) = f_1(g_1(I)) \cap \dotsc \cap f_i(g_i(I)).
  \]
  It is routine to verify that $I \setminus g(I) \subseteq f(g(I))$ and that $e\big(\HH_j[f(g(I))]\big) \le \eps e(\HH_j)$ for every $j \in [i]$.
\end{proof}

\subsection{Supersaturation results}
\label{sec:supersaturation-results}

The following two statements can be proved using a standard averaging argument in the spirit of the classical supersaturation theorem of Erd\H{o}s and Simonovits~\cite{ES}.

\begin{lemma}
  \label{lem:satLem-easy}
  Given graphs $H$ and $T$ and a $\delta > 0$, there exists an $\eps > 0$ such that the following holds. Every $n$-vertex graph $G$ with $\N_T(G) \ge \ex(n, T, H) + \delta n^{v_T}$ contains more than $\eps n^{v_H}$ copies of $H$.
\end{lemma}

\begin{lemma}
  \label{lem:satLem}
  Given graphs $H$ and $T$, a (finite) family $\F$ of $T$-coverings of $H$, and a $\delta > 0$, there exists an $\eps > 0$ such that the following holds. For every collection $\T \subseteq T(K_n)$ with $|\T| \ge \exx(n, T, \F) + \delta n^{v_T}$, there exists an $F \in \F$ such that $\T$ contains more than $\eps n^{v_{U(F)}}$ copies of $F$.
\end{lemma}

Our final lemma states that the extremal function $\exx(n, T, \F)$ corresponding to a family $\F$ of $T$-coverings of $H$ can be approximated by $\ex(n, T, H)$ at least when $\F$ contains the special $T$-covering $F_{T,H}^e$ of $H$ with $e_H$ copies of $T$.

\begin{lemma}
  \label{lem:Fe-exnTH}
  Given graphs $H$ and $T$, let $F^e = F_{T,H}^e$ be the $T$-covering of $H$ with $e_H$ copies of $T$ defined in Section~\ref{sec:notat-defin}. Then
  \[
    \exx(n, T, F^e)=\ex(n,T,H) +o(n^{v_T}).
  \]
\end{lemma}
\begin{proof}
  Since the underlying graph of $F^e$ contains a copy of $H$, then $\exx(n, T, F^e) \ge \N_T(G)$ for every $H$-free graph $G$. This shows that $\exx(n, T, F^e) \ge \ex(n,T,H)$. For the opposite inequality, fix an arbitrary $\eps > 0$ and suppose that $\T$ is a collection of $\ex(n,T,H) + \eps n^{v_T}$ copies of $T$ in $K_n$. Let $E$ be the set of all edges of $K_n$ that belong to fewer than $\eps n^{v_T-2}$ copies of $T$ from $\T$ and let $\T'$ comprise only those copies of $T$ from $\T$ that contain no edge from $E$. Observe that
  \[
    |\T'| \ge |\T| - |E| \cdot \eps n^{v_T-2} \ge |\T| - \binom{n}{2} \cdot \eps n^{v_T-2} > \ex(n,T,H).
  \]
  Let $G \subseteq K_n$ be the union of all copies of $T$ in $\T'$. Since $\N_T(G) \ge |\T'| > \ex(n,T, H)$, the graph $G$ contains a copy of $H$. As each edge of $G$ is contained in at least $\eps n^{v_T-2}$ copies of $T$ from $\T$, each copy of $H$ in $G$ must be covered by a copy of $F^e$ that is contained in $\T$. Indeed, given a copy of $H$ in $G$, one may construct such an $F^e$ greedily by considering the edges of $H$ ordered arbitrarily as $f_1, \dotsc, f_{e_H}$ and then finding some $T_i \in \T$ that contains $f_i$ and whose remaining $v_T-2$ vertices lie outside of $V(T_1) \cup \dotsc \cup V(T_{i-1})$, for each $i \in \{1, \dotsc, e_H\}$ in turn. One is guaranteed to find such a $T_i$ since the number of copies of $T$ in $T(K_n)$ that contain $f_i$ and have at least one more vertex in $V(T_1) \cup \dotsc \cup V(T_{i-1})$ is only $O(n^{v_T-3})$.
\end{proof}

\subsection{Properties of graph densities}

Here, we establish several useful facts relating the three notions of graph density that we consider in this work -- the density, the $2$-density, and the $T$-density. Our first lemma partially explains why the two cases $m_2(H) \le m_2(T)$ and $m_2(H) > m_2(T)$, which we consider separately while studying the typical value of $\ex\big(G(n,p), T, H\big)$, are so different.

\label{sec:graph-densities}

\begin{lemma}
  \label{lem:mTH-m2T}
  Suppose that $H$ and $T$ are fixed graphs and assume that $T$ is $2$-balanced.
  \begin{enumerate}[label=(\roman*),itemsep=3pt]
  \item
    \label{item:mH-leq-mT}
    If $m_2(H) \le m_2(T)$, then the $T$-covering $F^e = F_{T,H}^e$ of $H$ with $e_H$ edges satisfies $m_T(F^e) \le m_2(T)$.
  \item
    \label{item:mH-gt-mT}
    If $m_2(H) > m_2(T)$, then every $T$-covering $F$ of $H$ satisfies $m_T(F) > m_2(T)$.
  \end{enumerate}
\end{lemma}
\begin{proof}
  To see~\ref{item:mH-leq-mT}, assume that $m_2(H) \le m_2(T)$ and fix some $F' \subseteq F^e$. Since $F'$ is a $T$-covering of some subgraph $H' \subseteq H$ with $|F'|$ edges by pairwise edge-disjoint copies of $T$, then
  \begin{multline*}
    \frac{e_{U(F')} - e_T}{v_{U(F')} - v_T} = \frac{e_{H'} e_T - e_T}{v_{H'} + e_{H'} (v_T-2) - v_T} = \frac{(e_{H'}-1)(e_T-1) + e_{H'}-1}{(e_{H'}-1)(v_T-2) + v_{H'}-2} \\
    \le \max\left\{\frac{e_T-1}{v_T-2}, \frac{e_{H'}-1}{v_{H'}-2}\right\} \le \max\{m_2(T), m_2(H)\} = m_2(T).
  \end{multline*}
  To see~\ref{item:mH-gt-mT}, assume that $m_2(H) > m_2(T)$ and let $F$ be an arbitrary $T$-covering of $H$. Let $H' \subseteq H$ be any subgraph of $H$ satisfying $\frac{e_{H'}-1}{v_{H'}-2} > \frac{e_T-1}{v_T-2}$ and denote by $T_1, \dotsc, T_k$ all those elements of $F$ that intersect $H'$. For each $i \in [k]$, denote by $v_i$ and $e_i$ the numbers of vertices and edges of $T_i \cap H'$, respectively, and note that $e_i-1 \le m_2(T) (v_i-2)$. One easily verifies that
  \begin{multline*}
    m_T(F) \ge \frac{e_{U(F')} - e_T}{v_{U(F')} - v_T} = \frac{e_{H'} + \sum_{i=1}^k (e_T-e_i) - e_T}{v_{H'} + \sum_{i=1}^k (v_T - v_i) - v_T } \\
    = \frac{e_{H'} - 1 + (k-1)(e_T-1) - \sum_{i=1}^k(e_i-1)}{v_{H'}-2 + (k-1)(v_T-2) - \sum_{i=1}^k (v_i-2)} > \frac{e_T-1}{v_T-2} = m_2(T),
  \end{multline*}
  as claimed.
\end{proof}

Our next lemma computes the rate of the lower tail of the number of copies of a $2$-balanced graph $T$ in $G(n,p)$, which Lemma~\ref{lem:NumberOfCopiesOfT}, stated below, provides in a somewhat implicit form.
\begin{lemma}
  \label{lemma:Psi-T}
  If $T$ is a $2$-balanced graph, then
  \[
    \min\left\{n^{v(T')} p^{e(T')} : \emptyset \neq T' \subseteq T \right\} =
    \begin{cases}
      n^{v_T} p^{e_T} & \text{if $p \le n^{-1/m_2(T)}$}, \\
      n^2p & \text{if $p \ge n^{-1/m_2(T)}$}.
    \end{cases}
  \]
\end{lemma}
\begin{proof}
  Let $T$ be a $2$-balanced graph. Suppose first that $p \ge n^{-1/m_2(T)}$ and fix some $T' \subseteq T$ with at least two edges. Since $T$ is $2$-balanced, then $m_2(T) \ge \frac{e_{T'}-1}{v_{T'}-2}$ and hence $p \ge n^{-\frac{v_{T'}-2}{e_{T'}-1}}$. It follows that
  \[
    n^{v_{T'}} p^{e_{T'}} = n^2 p \cdot n^{v_{T'}-2} p^{e_{T'}-1} \ge n^2p \cdot n^{v_{T'}-2} \left(n^{-\frac{v_{T'}-2}{e_{T'}-1}}\right)^{e_{T'}-1} = n^2p.
  \]
  Suppose now that $p \le n^{-1/m_2(T)}$ and fix a nonempty $T' \subseteq T$. Since $m_2(T) \ge \frac{e_{T'}-1}{v_{T'}-2}$, then
  \[
    \frac{e_T-e_{T'}}{m_2(T)} = \frac{(e_T-1)-(e_{T'}-1)}{m_2(T)} \ge (v_T-2) - (v_{T'}-2) = v_T-v_{T'}.
  \]
  It follows that
  \begin{equation}
    \label{eq:Psi-T-sharpness}
    n^{v_{T'}}p^{e_{T'}} = n^{v_T}p^{e_T} \cdot n^{v_{T'}-v_T} p^{e_{T'}-e_T} \ge n^{v_T} p^{e_T}  \cdot n^{v_{T'}-v_T} \left(n^{1/m_2(T)}\right)^{e_T-e_{T'}} \ge n^{v_T}p^{e_T},
  \end{equation}
  as required.
\end{proof}

\subsection{Small subgraphs in $G(n,p)$}
\label{sec:small-subgraphs-Gnp}

Our proofs will require several properties of the distribution of the number of copies of a given fixed graph $T$ in the random graph $G_{n,p}$. Following the classical approach of Ruci\'nski and Vince~\cite{RuVi85}, we first prove that if $T$ is $2$-balanced, then the number of copies of $T$ in $G_{n,p}$ is concentrated around its expectation, provided that this expectation tends to infinity with $n$. Moreover, we show that if $p \ll n^{-1/m_2(T)}$, then copies of $T$ in $G_{n,p}$ are essentially pairwise edge-disjoint.

\begin{lemma}
  \label{lem:edgeCopiesT}
  Suppose that $T$ is a fixed $2$-balanced graph, assume that $p \gg n^{-v_T/e_T}$, and let $G \sim G(n,p)$. Then w.h.p.\ $\N_T(G) = (1+o(1)) \cdot \E[\N_T(G)]$. Moreover, if $p \ll n^{-1/m_2(T)}$, then w.h.p.\ $G$ contains a subgraph $G^*$ with the following two properties:
  \begin{enumerate}[label=(\roman*),itemsep=2pt]
  \item
    $\N_T(G^*) = (1+o(1)) \cdot \E[\N_T(G)]$.
  \item
    Every edge of $G^*$ belongs to exactly one copy of $T$.
  \end{enumerate}
\end{lemma}
\begin{proof}
  Assume that $p \gg n^{-v_T/e_T}$ and let $G \sim G(n,p)$. Let $X = \N_T(G)$ and write $Y$ for the number of pairs of distinct copies of $T$ in $G$ that share at least one edge. A routine calculation (see, e.g., \cite[Chapter~3]{JLR-book}) shows that
  \[
    \Var(X) \le \E[X] + \E[Y] \quad \text{and} \quad \E[Y] \le C \cdot \E[X]^2 \cdot \left(\min\left\{n^{v(T')} p^{e(T')} : \emptyset \neq T' \subsetneq T \right\}\right)^{-1}
  \]
  for some constant $C$ that depends only on $T$. Since $\E[X] = \Theta\big(n^{v_T} p^{e_T}\big)$, our assumption on $p$ implies that $\E[X] \to \infty$ and, by Lemma~\ref{lemma:Psi-T}, that $\Var(X) \ll \E[X]^2$. This proves the first assertion of the lemma. To see the second assertion, suppose further than $p \ll n^{-1/m_2(T)}$. We claim that in this case,
  \[
    \min\left\{n^{v(T')} p^{e(T')} : \emptyset \neq T' \subsetneq T \right\} \gg n^{v_T} p^{e_T}.
  \]
  To see this, one can repeat the calculation in the proof of Lemma~\ref{lemma:Psi-T} observing that under the assumption that $p \ll n^{-1/m_2(T)}$, the first `$\ge$' in~\eqref{eq:Psi-T-sharpness} can be replaced with a `$\gg$' (because $e_{T'} < e_T$). This means, in particular, that $\E[Y] \ll \E[X]$ and thus w.h.p.\ $Y \ll X$. Finally, observe that if $X = (1+o(1)) \E[X]$ and $Y \ll X$, then one may obtain a graph $G^*$ with the claimed properties by first removing from $G$ all edges that belong to more than one copy of $T$ and subsequently removing all edges that are not contained in any copy of $T$.
\end{proof}

The following optimal tail estimate for the number of copies of a fixed graph $T$ from a given family $\T \subseteq T(K_n)$ that appear in $G(n,p)$ is a rather straightforward extension of the result of Janson, \L uczak, and Ruci\'nski~\cite{JLR}.

\begin{lemma}
  \label{lem:NumberOfCopiesOfT}
  For every graph $T$ and constant $\delta > 0$, there exists a constant $\beta > 0$ such that the following holds. For every $p$ and each collection $\T$ of copies of $T$ in $K_n$,
  \[
    \Pr\Big(\big|\T \cap T\big(G(n,p)\big)\big| \le \big(|\T| - \delta n^{v_T}\big) \cdot p^{e_T} \Big) \le \exp\left(- \beta \cdot \min\big\{n^{v_{T'}}p^{e_{T'}} : \emptyset \neq T' \subseteq T\big\}\right).
  \]
  In particular, if $T$ is $2$-balanced, then
  \[
    \Pr\Big(\big|\T \cap T\big(G(n,p)\big)\big| \le \big(|\T| - \delta n^{v_T}\big) \cdot p^{e_T} \Big) \le
    \begin{cases}
      \exp\left( -\beta n^{v_T} p^{e_T}\right) & \text{if $p \le n^{-1/m_2(T)}$}, \\
            \exp\left( -\beta n^2p\right) & \text{if $p \ge n^{-1/m_2(T)}$}.
    \end{cases}
  \]
\end{lemma}

Note that the second assertion of the lemma follows immediately from the main assertion and Lemma~\ref{lemma:Psi-T}. We shall derive Lemma~\ref{lem:NumberOfCopiesOfT} from the following well-known inequality (see, for example, \cite[Chapter~8]{AS}).

\begin{theorem}[Janson's inequality]
  \label{thm:Janson}
  Suppose that $\Omega$ is a finite set and let $B_1, \dotsc, B_k$ be arbitrary subsets of $\Omega$. Form a random subset $R \subseteq \Omega$ by independently keeping each $\omega \in \Omega$ with probability $p_\omega \in [0,1]$. For each $i \in [k]$, let $X_i$ be the indicator of the event that $B_i \subseteq R$. Let $X = \sum_i X_i$ and define
  \[
    \mu = \E[X] = \sum_{i=1}^k \prod_{\omega \in B_i} p_\omega \qquad \text{and} \qquad \Delta = \sum_{\substack{i \neq j \\ B_i \cap B_j \neq \emptyset}} \E[X_i X_j] = \sum_{\substack{i \neq j \\  B_i \cap B_j \neq \emptyset}} \prod_{\omega \in B_i \cup B_j} p_\omega.
  \]
  Then for any $0 \le t \le \mu$,
  \[
    \Pr\big( X\le \mu-t \big) \le \exp \left(-\frac{t^2}{2(\mu+\Delta)}\right).
  \]
\end{theorem}

\begin{proof}[Proof of Lemma~\ref{lem:NumberOfCopiesOfT}]
  Suppose that $\T = \{T_1, \dotsc ,T_k\}$ and for each $i \in [k]$, let $X_i$ be the indicator of the event that $T_i$ appears in $G(n, p)$, so that
  \[
    X = \sum_{i=1}^k X_i = \big|\T \cap T\big(G(n,p)\big)\big|.
  \]
  Let $\mu$ and $\Delta$ be as in the statement of Theorem~\ref{thm:Janson} and observe that
  \[
    \mu  = \E[X] = |\T| \cdot p^{e_T} \le n^{v_T} p^{e_T}
  \]
  and that
  \[
    \begin{split}
      \Delta & = \sum_{i=1}^k \sum_{\substack{j \neq i \\ T_i \cap T_j \neq \emptyset}} \Pr\big(T_i \cup T_j \subseteq G(n,p)\big) \le |\T| \cdot \sum_{\emptyset \neq T' \subsetneq T}n^{v_T-v_{T'}}p^{2e_T-e_{T'}}\\
      & \le 2^{e_T}n^{2v_T} p^{2e_T} \cdot \left(\min\big\{n^{v_{T'}}p^{e_{T'}} : \emptyset \neq T' \subsetneq T\big\}\right)^{-1}.
    \end{split}
  \]
  It thus follows from Theorem~\ref{thm:Janson} that
  \[
    \begin{split}
      \Pr\big(X \le \mu - \delta n^{v_T}p^{e_T}\big) & \le \exp\left(-\frac{\delta^2n^{2v_T}p^{2e_T}}{2(\mu+\Delta)}\right) \le \exp\left(-\delta^2n^{2v_T}p^{2e_T} \cdot \min\left\{\frac{1}{4\mu}, \frac{1}{4\Delta}\right\}\right) \\
      & \le \exp\left(-2^{-e_T-2} \delta^2 \cdot \min\big\{n^{v_{T'}}p^{e_{T'}} : \emptyset \neq T' \subseteq T\big\}\right),
    \end{split}
  \]
  as claimed.
\end{proof}

\subsection{Harris's inequality}
\label{sec:harris-inequality}

Our proofs of Theorems~\ref{thm:main-easy} and~\ref{thm:main} will use the well-known correlation inequality due to Harris~\cite[Lemma~4.1]{Har60}. Below, we state a version of this inequality that is a slight rephrasing of~\cite[Theorem~6.3.2]{AS}. A family $\GG$ of graphs is called \emph{decreasing} if for every $G \in \GG$, every subgraph of $G$ belongs to $\GG$. A family $\GG$ of subgraphs of $K_n$ is called \emph{increasing} if for every $G \in \GG$, every $H \subseteq K_n$ such that $H \supseteq G$ also belongs to $\GG$.

\begin{theorem}
  \label{thm:Harris}
  Let $\GG_1$ and $\GG_2$ be two families of subgraphs of $K_n$ and suppose that $G \sim G(n,p)$. If $\GG_1$ is decreasing and $\GG_2$ is increasing, then
  \[
    \Pr(G \in \GG_1 \text{ and } G \in \GG_2) \le \Pr(G \in \GG_1) \cdot \Pr(G \in \GG_2).
  \]
\end{theorem}

\section{Proof of Theorems~\ref{thm:main-easy} and~\ref{thm:main}}
\label{sec:proof-main-thms}

\subsection{Proof of Theorem~\ref{thm:main-easy}}
\label{sec:proof-theorem-main-easy}

Suppose that $H$ and $T$ are fixed graphs and assume that $T$ is $2$-balanced and that $m_2(H) > m_2(T)$.

\begin{proof}[{Proof of the first assertion}]
  Suppose that $n^{-v_T/e_T} \ll p \ll n^{-1/m_2(H)}$ and let $G \sim G(n,p)$. It follows from Lemma~\ref{lem:edgeCopiesT} that w.h.p.\ $\N_T(G) = (1+o(1))\E[\N_T(G)] = (1+o(1)) \N_T(K_n) p^{e_T}$. Therefore, it suffices to show that for every positive constant $\delta$, w.h.p.\ $G$ contains an $H$-free subgraph with at least $\big(\N_T(K_n) - \delta n^{v_T}\big) \cdot p^{e_T}$ copies of $T$. We shall argue somewhat differently depending on whether or not $p \ll n^{-1/m_2(T)}$.

  \paragraph{Case 1. $p \ll n^{-1/m_2(T)}$.}

  Suppose that $G$ satisfies both assertions of Lemma~\ref{lem:edgeCopiesT} and let $G^*$ be the subgraph of $G$ from the statement of the lemma. Since each edge of $G^*$ is contained in exactly one copy of $T$, then each copy of $H$ in $G^*$ must correspond to some $T$-covering of $H$ in $T(G^*)$.\footnote{Recall that $T$-coverings of $H$ are collections of pairwise edge-disjoint copies of $T$.} Consider an arbitrary $T$-covering $F$ of $H$. Since we have assumed that $m_2(H) > m_2(T)$, part~\ref{item:mH-gt-mT} of Lemma~\ref{lem:mTH-m2T} yields $m_T(F) > m_2(T)$. Since $p \ll n^{-1/m_2(T)} \ll n^{-1/m_T(F)}$, Remark~\ref{remark:m-T-F} implies that there is some $F' \subseteq F$ such that
  \[
    \E\big[\N_{U(F')}(G)\big] \ll \E\big[\N_T(G)\big].
  \]
  Since there are only $O(1)$ types of $T$-coverings of $H$, then w.h.p.\ one may remove from $G^*$ some $o\big(\E[\N_T(G)]\big)$ edges to obtain an $H$-free graph $G_0$. Since each edge of $G^*$ belongs to exactly one copy of $T$, then
  \[
    \N_T(G_0) = \N_T(G^*) - o\big(\E[\N_T(G)]\big) \ge (1+o(1)) \cdot \N_T(K_n) p^{e_T}.
  \]

  \paragraph{Case 2. $p = \Omega(n^{-1/m_2(T)})$.} 

  Suppose that $G$ satisfies the assertion of Lemma~\ref{lem:edgeCopiesT}. Since $p \ll n^{-1/m_2(H)}$, then there is some $H' \subseteq H$ such that $\E[\N_{H'}(G)] \ll n^2p$. In particular, w.h.p.\ one may delete $o(n^2 p)$ edges from $G$ to make it $H$-free. It suffices to show that w.h.p.\ for every set $X$ of $o(n^2p)$ edges of $G$, the graph $G \setminus X$ contains at least $\big(\N_T(K_n) - \delta n^{v_T}\big) \cdot p^{e_T}$ copies of $T$. For a fixed $X \subseteq E(K_n)$, let $\cA_X$ denote the event that
  \[
    \N_T(G \setminus X) \le \big(\N_T(K_n) - \delta n^{v_T}\big) \cdot p^{e_T}.
  \]
  Since $|X| \ll n^2$, then $\N_T(K_n \setminus X) = \N_T(K_n) - o(n^{v_T})$ and thus Lemma~\ref{lem:NumberOfCopiesOfT} with $\T \leftarrow T(K_n \setminus X)$ together with Lemma~\ref{lemma:Psi-T} yield
  \[
    \Pr(\cA_X) \le \exp\left(-\beta n^2 p\right)
  \]
  for some positive constant $\beta$. Since for every $X \subseteq E(K_n)$, the event $X \subseteq G$ is increasing and the event $\cA_X$ is decreasing, Harris's inequality (Theorem~\ref{thm:Harris}) implies that
  \[
    \Pr\big(X \subseteq G \text{ and } \cA_X\big) \le \Pr\big(X \subseteq G\big) \cdot \Pr(\cA_X).
  \]
  Consequently,
  \begin{multline*}
    \Pr\big(\cA_X \text{ for some } X \subseteq G \text{ with $|X| = o(n^2p)$}\big) \le \sum_{\substack{X \subseteq E(K_n) \\ |X| \ll n^2p}} p^{|X|} \cdot \exp\left(-\beta n^2p\right) \\
    \le \sum_{x \ll n^2p} \binom{\binom{n}{2}}{x} p^x \cdot \exp\left(-\beta n^2 p\right) \le \sum_{x \ll n^2p} \left(\frac{en^2p}{2x}\right)^x \cdot \exp\left(-\beta n^2 p\right) \le \exp\left(-\beta n^2 p/2\right),
  \end{multline*}
  as the function $x \mapsto (ea/x)^x$ is increasing when $x \le a$.
\end{proof}

\begin{proof}[{Proof of the second assertion}]
  Suppose that $p \gg n^{-1/m_2(H)}$ and let $G \sim G(n,p)$. Our aim is to show that for every positive constant $\delta$, w.h.p.\ every $H$-free subgraph $G_0$ of $G$ satisfies
  \[
    \N_T(G_0) \le \big(\ex(n, T, H) + 2\delta n^{v_T}\big) \cdot p^{e_T}.
  \]
  
  Let $\HH$ be the $e_H$-uniform hypergraph with vertex set $E(K_n)$ whose edges are all copies of $H$ in $K_n$. Observe that
  \[
    v(\HH) = \Theta\big(n^2\big) \qquad \text{and} \qquad e(\HH) = \Theta\big(n^{v_H}\big)
  \]
  and that there is a natural one-to-one correspondence between the independent sets of $\HH$ and $H$-free subgraphs of $K_n$. As we shall be applying Theorem~\ref{thm:Cont} to the hypergraph $\HH$, we let $q = n^{-1/m_2(H)}$ and verify that $\HH$ satisfies the main assumption of the theorem, provided that $K$ is a sufficiently large constant.

\begin{claim}
  There is a constant $K$ such that the hypergraph $\HH$ satisfies~\eqref{eq:Delta-ell-ass} in Theorem~\ref{thm:Cont} with $q=n^{-1/m_2(H)}$.
\end{claim}

\begin{proof}
  Fix an arbitrary $\ell \in [e_H]$ and note that $\Delta_\ell(\HH)$ is the largest number of copies of $H$ in $K_n$ that contain some given set of $\ell$ edges. It follows that
  \[
    \Delta_\ell(\HH) \le \sum_{H' \subseteq H, e_{H'} = \ell} n^{v_{H} - v_{H'}}
  \]
  and hence
  \[
    \frac{v(\HH)}{e(\HH)} \cdot \max_{\ell \in [e_H]} \frac{\Delta_\ell(\HH)}{q^{\ell-1}} \le 2^{e_H} \cdot n^{2-v_H} \cdot \max_{\emptyset \neq H' \subseteq H} \frac{n^{v_H - v_{H'}}}{q^{e_{H'}-1}} = 2^{e_H} \cdot \left( \min_{\emptyset \neq H' \subseteq H} n^{v_{H'} - 2}q^{e_{H'} - 1} \right)^{-1}.
  \]
  Finally, since $q = n^{-1/m_2(H)}$, then $n^{v_{H'} - 2}q^{e_{H'} - 1} \ge 1$ for every nonempty $H' \subseteq H$.
\end{proof}

Denote by $\Free_n(H)$ the family of all $H$-free subgraphs of $K_n$ and let $\eps$ be the constant given by Lemma~\ref{lem:satLem-easy} invoked with $\delta/4$ in place of $\delta$. Apply Theorem~\ref{thm:Cont} to the hypergraph $\HH$ to obtain a constant $C$, a family $\cS \subseteq \binom{E(K_n)}{\le Cqn^2}$, and functions $g \colon \Free_n(H) \to \cS$ and $f \colon \cS \to \cP(E(K_n))$ such that:
\begin{enumerate}[label=(\roman*),itemsep=3pt]
\item
  For every $G_0 \in \Free_n(H)$, $g(G_0) \subseteq G_0$ and $G_0 \setminus g(G_0) \subseteq  f(g(G_0))$.
\item
  \label{item:fS-sparse-easy}
  For every $S\in \cS$, the graph $f(S)$ contains at most $\eps n^{v_H}$ copies of $H$.
\end{enumerate}

Given an $S \in \cS$, denote by $\cA_S$ the event
\[
  \big| T(G) \setminus T\big(f(S) \cup S\big) \big| \le \big( \N_T(K_n) - \ex(n, T, H) - \delta n^{v_T} \big) \cdot p^{e_T}.
\]

\begin{claim}
  \label{claim:NotToManyT-easy}
  There is a constant $\beta > 0$ such that for every $S \in \cS$,
  \[
    \Pr(\cA_S) \le \exp\left(-\beta n^2p\right).
  \]
\end{claim}
\begin{proof}
  Fix an $S \in \cS$ and let $\T_S$ denote the collection of all copies of $T$ in $K_n$ that are not completely contained in $f(S) \cup S$. Since $|S| \ll n^2$, then property~\ref{item:fS-sparse-easy} above and Lemma~\ref{lem:satLem-easy} imply that
  \[
    \begin{split}
      |\T_S| & = \N_T(K_n) - \N_T\big(f(S) \cup S\big) \ge \N_T(K_n) - \N_T\big(f(S)\big) - |S| \cdot n^{v_T-2} \\
      & \ge \N_T(K_n) - \ex(n, T, H) - \delta n^{v_T}/2.
    \end{split}
  \]
  Since $T$ is $2$-balanced and $p \gg n^{-1/m_2(H)} \ge n^{-1/m_2(T)}$, Lemma~\ref{lem:NumberOfCopiesOfT} implies that
  \[
    \Pr(\cA_S) \le \exp\left(-\beta n^2 p\right)
  \]
  for some positive constant $\beta$, as claimed.
\end{proof}

Suppose that $G$ satisfies the assertion of Lemma~\ref{lem:edgeCopiesT} and let $G_0 \subseteq G$ be an $H$-free subgraph of $G$ that maximizes $\N_T(G_0)$. Since $G_0 \in \Free_n(H)$, then
\[
  g(G_0) \subseteq G_0 \subseteq f(g(G_0)) \cup g(G_0).
\]
and hence
\begin{align*}
  \N_T(G_0) & \le \max\left\{\left|T(G) \cap T\big(f(S) \cup S\big)\right| : S \in \cS \text{ and } S \subseteq G \right\} \\
  & = \N_T(G) - \min\left\{\left|T(G) \setminus T\big(f(S) \cup S\big)\right| : S \in \cS \text{ and } S \subseteq G \right\} \\
  & = (1+o(1)) \cdot \E[\N_T(G)] - \min\left\{\left|T(G) \setminus T\big(f(S) \cup S\big)\right| : S \in \cS \text{ and } S \subseteq G \right\}.
\end{align*}

We shall show that w.h.p.\ $\cA_S$ does not hold for any $S \in \cS$ such that $S \subseteq G$, which will imply that
\[
  \N_T(G_0) \le \big(\ex(n, T, H) + 2\delta n^{v_T}\big) \cdot p^{e_T}.
\]
Since for every $S \in \cS$, the event $S \subseteq G$ is increasing and the event $\cA_S$ is decreasing, Harris's inequality (Theorem~\ref{thm:Harris}) implies that
\[
  \Pr\big(S \subseteq G \text{ and } \cA_S\big) \le \Pr\big(S \subseteq G\big) \cdot \Pr(\cA_S).
\]
By Claim~\ref{claim:NotToManyT-easy}, in order to complete the proof in this case, it is sufficient to prove the following.

\begin{claim}
  \label{lem:boundOfSig-easy}
  \[
    \sum_{S \in \cS} \Pr\big(S \subseteq G\big) \le \exp\left(o(n^2p)\right).
  \]
\end{claim}
\begin{proof}
  Since each $S \in \cS$ is a graph with at most $Cqn^2$ edges and $q = n^{-1/m_2(H)} \ll p$, then
  \[
      \sum_{S \in \cS} \Pr\big(S \subseteq G\big) \le \sum_{s \le Cqn^2} \binom{n^2}{s} p^s \le \sum_{s = o\left(pn^2\right)} \left(\frac{e n^2 p}{s} \right)^s = \exp\left(o\big(n^2p\big)\right),
  \]
  as the function $s \mapsto (ea/s)^s$ is increasing when $s \le a$.
\end{proof}
This completes the proof of the second assertion of Theorem~\ref{thm:main-easy}.
\end{proof}

\subsection{Proof of Theorem~\ref{thm:main}}

\label{sec:proof-theorem-main}

Suppose that $H$ and $T$ are fixed graphs and assume that $T$ is $2$-balanced and that $m_2(H) \le m_2(T)$. Recall Definition~\ref{def:T-resolution}, let $F_1, \dotsc, F_k$ be the $T$-resolution of $H$, and let $p_0, p_1, \dotsc, p_k$ be the associated threshold sequence. For each $i \in \{0, \dotsc, k\}$, denote by $\F_i$ the set $\{F_1, \dotsc, F_i\}$. Finally, let $F^e = F_{T,H}^e$ be the minimal covering of $H$ with $e_H$ pairwise edge-disjoint copies of $T$.

\begin{proof}[{Proof of part~\ref{item:main-0-statement}}]
  Fix an $i \in [k]$, suppose that $p_0 \ll p \ll p_i$, and let $G \sim G(n,p)$. Our aim is to show that for every positive constant $\delta$, w.h.p.\ $G$ contains an $H$-free subgraph with at least $\left(\exx\big(n, T, \F_{i-1}\big) - \delta n^{v_T}\right) \cdot p^{e_T}$ copies of $T$. If $\exx\big(n, T, \F_{i-1}\big) = o\big(n^{v_T}\big)$, then the assertion is trivial (we may simply take the empty graph), so for the remainder of the proof we shall assume that $\exx\big(n, T, \F_{i-1}\big) \ge \gamma n^{v_T}$ for some positive constant $\gamma$.

  It follows from part~\ref{item:mH-leq-mT} of Lemma~\ref{lem:mTH-m2T} that $p \ll p_i \le n^{-1/m_T(F^e)} \le n^{-1/m_2(T)}$, so we may assume that $G$ satisfies both assertions of Lemma~\ref{lem:edgeCopiesT}. Let $G^*$ be the subgraph of $G$ from the statement of the lemma and let $\T_{i-1}$ be an extremal collection of copies of $T$ in $K_n$ with respect to being $\F_{i-1}$-free. In other words, let $\T_{i-1}$ be a collection of $\exx(n,T,\F_{i-1})$ copies of $T$ that does not contain any $T$-covering of either of the types $F_1, \dotsc, F_{i-1}$. Let $G'$ be the graph obtained from $G^*$ by keeping only edges covered by $T(G^*) \cap \T_{i-1}$ and let $G_0$ be the graph obtained from $G'$ by deleting all edges from every copy of $H$ in $G'$. This graph is clearly $H$-free. Since each edge of $G^*$ is contained in exactly one copy of $T$, then each copy of $H$ in $G^*$ must belong to some $T$-covering of $H$. Since $\T_{i-1}$ is $\F_{i-1}$-free, then the only $T$-coverings of $H$ that we may find in $G'$ are $F_i, \dotsc, F_k$ and coverings whose $T$-density is strictly greater than $m_T(F^e)$. Since $p \ll p_i \le n^{-1/m_T(F^e)}$ and there are only $O(1)$ types of $T$-coverings, then w.h.p.\ there are only $o\big(\E[\N_T(G)]\big)$ edges in $G' \setminus G_0$ and thus $\N_T(G') - \N_T(G_0) = o\big(\E[\N_T(G)]\big)$, as every edge of $G'$ belongs to at most one copy of $T$. Now, Lemma~\ref{lem:NumberOfCopiesOfT} implies that w.h.p.
  \[
    |T(G) \cap \T_{i-1}| \ge \left(\exx\big(n,T,\F_{i-1}\big) - \delta n^{v_T}/3\right) \cdot p^{e_T}.
  \]
  Therefore,
  \begin{align*}
    \N_T(G_0) & \ge \N_T(G') - \delta n^{v_T} p^{e_T} / 3 = |T(G^*) \cap \T_{i-1}| - \delta n^{v_T} p^{e_T} / 3 \\
    & \ge |T(G) \cap \T_{i-1}| - \left(\N_T(G) - \N_T(G^*)\right) - \delta n^{v_T} p^{e_T} / 3 \ge \left(\exx\big(n,T,\F_{i-1}\big) - \delta n^{v_T}\right) \cdot p^{e_T},
  \end{align*}
  since $\N_T(G) = \N_T(G^*) + o\big(n^{v_T} p^{e_T}\big)$.
\end{proof}

\begin{proof}[{Proof of part~\ref{item:main-1-statement}}]
  Fix an $i \in [k]$, suppose that $p \gg p_i$, and let $G \sim G(n,p)$. Our aim is to show that for every positive constant $\delta$, w.h.p.\ every $H$-free subgraph $G_0$ of $G$ satisfies
  \begin{equation}\label{eq:upperBound}
    \N_T(G_0) \le \left(\exx\big(n, T, \F_i\big) + 2\delta n^{v_T}\right) \cdot p^{e_T}.
  \end{equation}
  
  For each $j \in [i]$, let $\HH_j$ be the $|F_j|$-uniform hypergraph whose vertices are all copies of $T$ in $K_n$ and whose edges are all collections of $|F_j|$ copies of $T$ in $K_n$ that are isomorphic to the $T$-covering $F_j$. Observe that
\[
    v(\HH_j) = \Theta\big(n^{v_T}\big) \qquad \text{and} \qquad e(\HH_j) = \Theta\big(n^{v_{U(F_j)}}\big).
\]
Since $U(F_j)$ contains a copy of $H$, then for every $H$-free graph $G_0$, the family $T(G_0)$ is an independent set in $\HH_j$, for each $j \in [i]$. As we shall be applying Corollary~\ref{cor:Cont} to the hypergraphs $\HH_1, \dotsc, \HH_i$, we let $q = p_i^{e_T}$ and verify that all $\HH_j$ satisfy the main assumption of the corollary, provided that $K$ is a sufficiently large constant.

\begin{claim}
  There is a constant $K$ such that for each $j \in [i]$, the hypergraph $\HH_j$ satisfies~\eqref{eq:Delta-ell-ass-cor} in Corollary~\ref{cor:Cont} with $q=p_i^{e_T}$.
\end{claim}

\begin{proof}
  Fix an arbitrary $\ell \in [|F_j|]$ and note that $\Delta_\ell(\HH_j)$ is the largest number of copies of $F_j$ in $T(K_n)$ that share the same set of $\ell$ copies of $T$. It follows that
  \[
    \Delta_\ell(\HH_j) \le \sum_{F' \subseteq F_j, |F'| = \ell} n^{v_{U(F_j)} - v_{U(F')}}
  \]
  and hence
  \begin{align*}
    \frac{v(\HH_j)}{e(\HH_j)} \cdot \max_{\ell \in [|F_j|]} \frac{\Delta_\ell(\HH_j)}{q^{\ell-1}} & \le 2^{|F_j|} \cdot \frac{n^{v_T}}{n^{v_{U(F_j)}}} \cdot \max_{\emptyset \neq F' \subseteq F_j} \frac{n^{v_{U(F_j)} - v_{U(F')}}}{p_i^{e_T \cdot (|F'|-1)}} \\
    & = 2^{|F_j|} \cdot \left( \min_{\emptyset \neq F' \subseteq F_j} n^{v_{U(F')} - v_T}p_i^{e_{U(F')} - e_T} \right)^{-1}.
  \end{align*}
  Finally, since $p_i \ge p_j = n^{-1/m_T(F_j)}$, then
  \[
    n^{v_{U(F')} - v_T}p_i^{e_{U(F')} - e_T} \ge 1
  \]
  for every nonempty $F' \subseteq F_j$.
\end{proof}

Denote by $\Free_n(\F_i)$ the family of all subfamilies of $T(K_n)$ that do not contain any $T$-covering isomorphic to one of the members of $\F_i$ and let $\eps$ be the constant given by Lemma~\ref{lem:satLem} invoked with $\delta/2$ in place of $\delta$. Apply Corollary~\ref{cor:Cont} to the hypergraphs $\HH_1, \dotsc, \HH_i$ to obtain a constant $C$, a family $\cS \subseteq \binom{T(K_n)}{\le Cqn^{v_T}}$, and functions $g \colon \Free_n(\F_i) \to \cS$ and $f \colon \cS \to \cP(T(K_n))$ such that:
\begin{enumerate}[label=(\roman*),itemsep=3pt]
\item
  For every $\T \in \Free_n(\F_i)$, $g(\T) \subseteq \T$ and $\T \setminus g(\T) \subseteq  f(g(\T))$.
\item
  \label{item:fS-sparse}
  For every $S\in \cS$, the collection $f(S)$ has at most $\eps n^{v_{U(F_j)}}$ copies of $F_j$ for every $j \in [i]$.
\item
  \label{item:g-consistency}
  If $g(\T)\subseteq \T'$ and $g(\T')\subseteq \T$ for some $\T, \T' \in \Free_n(\F_i)$, then $g(\T) = g(\T')$.
\end{enumerate}

Given an $S \in \cS$, denote by $\cA_S$ the event
\[
  \big| T(G) \setminus f(S) \big| \le \big( \N_T(K_n) - \exx(n, T, \F_i) - \delta n^{v_T} \big) \cdot p^{e_T}.
\]

\begin{claim}
    \label{claim:NotToManyT}
  There is a constant $\beta > 0$ such that for every $S \in \cS$,
  \[
    \Pr(\cA_S) \le \exp\left(-\beta \cdot \min\left\{n^2p, n^{v_T} p^{e_T}\right\}\right).
  \]
\end{claim}
\begin{proof}
  Fix an $S \in \cS$ and let $\T_S$ denote the collection of all copies of $T$ in $K_n$ that do not belong to $f(S)$. Property~\ref{item:fS-sparse} above and Lemma~\ref{lem:satLem} imply that
  \[
    |\T_S| = \N_T(K_n) - |f(S)| \ge \N_T(K_n) - \exx(n, T, \F_i) - \delta n^{v_T}/2.
  \]
  Since $T$ is $2$-balanced, Lemma~\ref{lem:NumberOfCopiesOfT} implies that
  \[
    \Pr(\cA_S) \le \exp\left(-\beta \cdot \min\left\{n^2 p, n^{v_T} p^{e_T}\right\}\right)
  \]
  for some positive constant $\beta$, as claimed.
\end{proof}

We shall now argue somewhat differently depending on whether or not $p \ll n^{-1/m_2(T)}$.

\paragraph{Case 1. $p \ll n^{-1/m_2(T)}$.}

Suppose that $G$ satisfies both assertions of Lemma~\ref{lem:edgeCopiesT} and let $G^*$ be the subgraph of $G$ from the statement of the lemma. Let $G_0 \subseteq G$ be an $H$-free subgraph of $G$ that maximizes $\N_T(G_0)$ and let $G' = G_0 \cap G^*$. Since
\[
  \N_T(G_0) \le \N_T(G') + \N_T(G) - \N_T(G^*)  = \N_T(G') + o\big(\E[\N_T(G)]\big) = \N_T(G') + o\big(n^{v_T} p^{e_T}\big),
\]
it suffices to show that~\eqref{eq:upperBound} holds with $G_0$ replaced by $G'$. Since $G'$ is $H$-free, then $T(G') \in \Free_n(\F_i)$ and hence
\[
  g\big(T(G')\big) \subseteq T(G') \subseteq f\big(g\big(T(G')\big)\big) \cup g\big(T(G')\big).
\]
But this means that
\begin{align*}
  \N_T(G') & \le \max\left\{|T(G) \cap f(S)| + |S| : S \in \cS \text{ and } S \subseteq T(G') \right\} \\
  & = \N_T(G) - \min\left\{|T(G) \setminus f(S)| - |S| : S \in \cS \text{ and } S \subseteq T(G') \right\} \\
  & \le (1+o(1)) \cdot \E[\N_T(G)] + Cp_i^{e_T}n^{v_T} - \min\left\{|T(G) \setminus f(S)| : S \in \cS \text{ and } S \subseteq T(G') \right\} \\
  & = (1+o(1)) \cdot \E[\N_T(G)] - \min\left\{|T(G) \setminus f(S)| : S \in \cS \text{ and } S \subseteq T(G') \right\}.
\end{align*}

Now, let $\cS'$ comprise all the sets of $T$-copies $S \in \cS$ that are pairwise edge-disjoint. Since $T(G')$ is a collection of pairwise edge-disjoint copies of $T$, then
\[
  \min\left\{|T(G) \setminus f(S)| : S \in \cS \text{ and } S \subseteq T(G') \right\} = \min\left\{|T(G) \setminus f(S)| : S \in \cS' \text{ and } S \subseteq T(G') \right\}.
\]
We shall now show that w.h.p.\ $\cA_S$ does not hold for any $S \in \cS'$ such that $S \subseteq T(G)$, which will imply that
\[
  \N_T(G') \le \left(\exx(n, T, \F') + 2\delta n^{v_T}\right) \cdot p^{e_T}.
\]
Since for every $S \in \cS$, the event $S \subseteq T(G)$ is increasing and the event $\cA_S$ is decreasing, Harris's inequality (Theorem~\ref{thm:Harris}) implies that
\[
  \Pr\big(S \subseteq T(G) \text{ and } \cA_S\big) \le \Pr\big(S \subseteq T(G)\big) \cdot \Pr(\cA_S).
\]
Since we have assumed that $p \ll n^{-1/m_2(T)}$, then Claim~\ref{claim:NotToManyT} and Lemma~\ref{lemma:Psi-T} imply that
\[
  \Pr(\cA_s) \le \exp\left(-\beta n^{v_T} p^{e_T}\right)
\]
and consequently, in order to complete the proof in this case, it is sufficient to prove the following.

\begin{claim}
  \label{lem:boundOfSig}
  \[
    \sum_{S \in \cS'} \Pr\big(S \subseteq T(G)\big) \le \exp\big(o(n^{v_T} p^{e_T})\big).
  \]
\end{claim}
\begin{proof}
  Since each $S \in \cS'$ consists of pairwise edge-disjoint copies of $T$, then
  \[
    \Pr\big(S \subseteq T(G)\big) = \Pr\big(U(S) \subseteq G\big) = p^{e_{U(S)}} = p^{e_T \cdot |S|}.
  \]
  Since $\cS'$ contains only sets of at most $Cp_i^{e_T}n^{v_T}$ copies of $T$ in $K_n$ and $p_i \ll p$, it now follows that
  \[
    \begin{split}
      \sum_{S \in \cS'} \Pr\big(S \subseteq T(G)\big) & = \sum_{S \in \cS'} p^{e_T \cdot |S|} \le \sum_{s \le Cp_i^{e_T}n^{v_T}} \binom{n^{v_T}}{s} p^{e_T \cdot s} \\
      & \le \sum_{s = o\left(p^{e_T} n^{v_T}\right)} \left(\frac{e n^{v_T} p^{e_T}}{s} \right)^s = \exp\left(o\big(n^{v_T} p^{e_T}\big)\right),
    \end{split}
  \]
  since the function $s \mapsto (ea/s)^s$ is increasing when $s \le a$.
\end{proof}

\paragraph{Case 2. $p = \Omega(n^{-1/m_2(T)})$.} 

Suppose that $G$ satisfies the assertion of Lemma~\ref{lem:edgeCopiesT} and let $G_0 \subseteq G$ be an $H$-free subgraph of $G$ that maximizes $\N_T(G_0)$. Since $G_0$ is $H$-free, then $T(G_0) \in \Free_n(\F_i)$ and hence
\[
  g\big(T(G_0)\big) \subseteq T(G_0) \subseteq f\big(g\big(T(G_0)\big)\big) \cup g\big(T(G_0)\big).
\]
Now, let $\cS''$ comprise all the sets of $T$-copies $S \in \cS$ that are of the form $g\big(T(G'')\big)$ for some $H$-free graph $G'' \subseteq K_n$ and observe that
\begin{align*}
  \N_T(G_0) & \le \max\left\{|T(G) \cap f(S)| + |S| : S \in \cS'' \text{ and } S \subseteq T(G) \right\} \\
  & = \N_T(G) - \min\left\{|T(G) \setminus f(S)| - |S| : S \in \cS'' \text{ and } S \subseteq T(G) \right\} \\
  & = (1+o(1)) \cdot \E[\N_T(G)] - \min\left\{|T(G) \setminus f(S)| : S \in \cS'' \text{ and } S \subseteq T(G) \right\}.
\end{align*}

Analogously to Case~1, we shall show that w.h.p.\ $\cA_S$ does not hold for any $S \in \cS''$ such that $S \subseteq T(G)$, which will imply that
\[
  \N_T(G_0) \le \left(\exx(n, T, \F_i) + 2\delta n^{v_T}\right) \cdot p^{e_T},
\]
as claimed. As before, since for every $S \in \cS$, the event $S \subseteq T(G)$ is increasing and the event $\cA_S$ is decreasing, Harris's inequality (Theorem~\ref{thm:Harris}) implies that
\[
  \Pr\big(S \subseteq T(G) \text{ and } \cA_S\big) \le \Pr\big(S \subseteq T(G)\big) \cdot \Pr(\cA_S).
\]
Since we have assumed that $p = \Omega\big(n^{-1/m_2(T)}\big)$, then Claim~\ref{claim:NotToManyT} and Lemma~\ref{lemma:Psi-T} imply that
\[
  \Pr(\cA_s) \le \exp\left(-\beta n^2p\right)
\]
and consequently, in order to complete the proof in this case, it is sufficient to prove the following.

\begin{claim}
  \label{lem:boundOfSig-case2}
  \[
    \sum_{S \in \cS''} \Pr\big(S \subseteq T(G)\big) \le \exp\big(o(n^2p)\big).
  \]
\end{claim}
\begin{proof}
 We claim that the function $U$ that maps a collection of copies of $T$ to its underlying graph is injective when restricted to $\cS''$. Indeed, suppose that $U\left(g\big(T(G_1)\big)\right) = U\left(g\big(T(G_2)\big)\right)$ for some $H$-free graphs $G_1$ and $G_2$. It follows that
\[
  g\big(T(G_1)\big) \subseteq T\left(U\left(g\big(T(G_1)\big)\right)\right) = T\left(U\left(g\big(T(G_2)\big)\right)\right) \subseteq T\left(U\big(T(G_2)\big)\right) = T(G_2)
\]
and, vice-versa,  $g\big(T(G_2)\big) \subseteq T(G_1)$. The consistency property of the function $g$, see~\ref{item:g-consistency} above, implies that $g\big(T(G_1)\big) = g\big(T(G_2)\big)$.

Since $p_i \le n^{-1/m_2(T)}$ by part~\ref{item:mH-leq-mT} of Lemma~\ref{lem:mTH-m2T}, then Lemma~\ref{lemma:Psi-T} implies that $n^{v_T}p_i^{e_T} \le n^2p_i$. In particular, each $S \in \cS''$ comprises at most $Cn^2p_i$ copies of $T$ and therefore $U(S)$ has at most $Ce_Tn^2p_i$ edges. Since the function $U$ is injective when restricted to $\cS''$, we may conclude that
  \[
    \begin{split}
      \sum_{S \in \cS''} \Pr\big(S \subseteq T(G)\big) & = \sum_{S \in \cS''} \Pr\big(U(S) \subseteq G\big)= \sum_{U \in U(\cS'')} \Pr\big(U \subseteq G\big) = \sum_{U \in U(\cS'')} p^{e_U} \\
      & \le \sum_{u \le Ce_Tn^2p_i} \binom{\binom{n}{2}}{u} p^u \le \sum_{s = o\left(pn^2\right)} \left(\frac{e n^2 p}{2u} \right)^u = \exp\left(o\big(n^2p\big)\right),
    \end{split}
  \]
  since the function $u \mapsto (ea/u)^u$ is increasing when $u \le a$.  
\end{proof}

This completes the proof of part~\ref{item:main-1-statement} of Theorem~\ref{thm:main}.
\end{proof}

\section{Concluding remarks and open questions}
\label{sec:concl-remarks}

In this paper, we have studied the random variable $\ex\big(G(n,p), T, H\big)$ that counts the largest number of copies of $T$ in an $H$-free subgraph of the binomial random graph $G(n,p)$. We restricted our attention to the case when $T$ is $2$-balanced; the case when $T$ is not $2$-balanced poses further challenges and we were not able to resolve it using our methods. The threshold phenomena associated with the variable $\ex\big(G(n,p), T, H\big)$ are quite different depending on whether or not the inequality $m_2(H) > m_2(T)$ holds:

\begin{enumerate}[label={(\roman*)},itemsep=4pt]
\item
  If $m_2(H) > m_2(T)$, then our Theorem~\ref{thm:main-easy} offers a natural generalization of a sparse random analogue of the Erd\H{o}s--Stone theorem that was proved several years ago by Conlon and Gowers~\cite{CG} and by Schacht~\cite{Sc}.
\item
  If $m_2(H) \le m_2(T)$, then the `evolution' of the random variable $\ex\big(G(n,p), T, H\big)$ as $p$ grows from $0$ to $1$ exhibits a more complex behavior. Our Theorem~\ref{thm:main} shows that there are several potential `phase transitions' and that the typical values of the variable between these phase transitions are determined by solutions to deterministic hypergraph Tur\'an-type problems which we were unable to solve in full generality.
\end{enumerate}

\noindent
There are several natural directions for further investigations that are suggested by this work:

\begin{itemize}[itemsep=4pt]
\item
  It would be interesting to study the variable $\ex\big(G(n,p), T, H\big)$ for general graphs $T$ and $H$, that is, without assuming that $T$ is $2$-balanced.
\item
  We have very little understanding of the Tur\'an-type problems related to $T$-coverings of $H$ that are described in Section~\ref{sec:notat-defin}, even in the case when $T$ is a complete graph. A concrete problem that we found the most interesting is stated as Question~\ref{question:main}. In short, we ask if there exists a pair of graphs $T$ and $H$ such that the variable $\ex\big(G(n,p), T, H\big)$ undergoes multiple `phase transitions'.
\item
  Given a family $\HH$ of graphs, one may more generally ask to study the random variable $\ex\big(G(n,p), T, \HH\big)$ that counts the largest number of copies of $T$ in a subgraph of $G(n,p)$ that is free of every $H \in \HH$. This problem is solved when $T = K_2$ and $\HH$ is finite, see~\cite[Theorem~6.4]{MorSamSax}, but not much is known, even in the deterministic case ($p=1$), when $T \neq K_2$.
\end{itemize}

\smallskip
\noindent
\textbf{Acknowledgment:} We are indebted to the two anonymous referees for their careful reading of the manuscript and many helpful suggestions. The second author thanks Orit Raz for helpful discussions.

\bibliographystyle{amsplain}
\bibliography{T-copies-H-free}

\end{document}